\documentclass[11pt,a4paper]{article}
\usepackage[affil-it]{authblk}
\usepackage[utf8]{inputenc}
\usepackage{amsmath, amssymb, mathrsfs, amsthm, amsfonts, stmaryrd}
\usepackage{latexsym}
\usepackage[pdftex]{hyperref}
\usepackage{a4wide}

\title{Computing the Varchenko Determinant of a Bilinear Form}

\author{Hery Randriamaro \thanks{Mathematisches Forschungsinstitut Oberwolfach \\ Schwarzwaldstraße 9-11, 77709 Oberwolfach, Germany \\ E-mail: \texttt{hery.randriamaro@outlook.com}}}

\newtheorem{theorem}{Theorem}[section]

\newtheorem{lemma}[theorem]{Lemma}
\newtheorem{proposition}[theorem]{Proposition}

\theoremstyle{definition}

\theoremstyle{remark}

\begin{document}

\maketitle

\begin{abstract}
\noindent The Varchenko determinant is the determinant of the bilinear form associated to a real hyperplane arrangement. We show that we can obtain the exact value of this determinant for certain hyperplane arrangements if we know the edges which are relevant.   

\bigskip 


\smallskip

\end{abstract}

\section{Introduction}

\noindent Let $x = (x_1, \dots, x_n)$ be a variable of the Euclidean space $\mathbb{R}^n$, and $a_1, \dots, a_n, b$ real coefficients such that $(a_1, \dots, a_n) \neq (0, \dots, 0)$. A hyperplane $H$ of $\mathbb{R}^n$ is a $(n-1)$--dimensional affine subspace $H := \{ x \in \mathbb{R}^n \ |\ a_1 x_1 + \dots + a_n x_n = b \}$. An arrangement of hyperplanes in $\mathbb{R}^n$ is a finite set of hyperplanes. The most known hyperplane arrangement is certainly the braid arrangement $\mathcal{B}_n = \big\{\{x \in \mathbb{R}^n \ |\ x_i-x_j=0\}\big\}_{1 \leq i < j \leq n}$. Hyperplane arrangement theory is currently a very active area of research, combining ideas from algebraic combinatorics, algebraic topology, and algebraic geometry. In the preface of their book \cite{OrTe}, Orlik and Terao wrote \emph{Arrangements are easily defined and may be enjoyed at levels ranging from the recreational to the expert, yet these simple objects lead to deep and beautiful results. Their study combines methods from many areas of mathematics and reveals unexpected connections.} The bilinear form of a hyperplane arrangement defined by Varchenko confirms their affirmation. The Varchenko determinant is the determinant of this bilinear form.

\smallskip

\noindent An edge of a hyperplane arrangement $\mathcal{A}$ is a nonempty intersection of some of its hyperplanes. Denote by $L(\mathcal{A})$ the set of all edges of $\mathcal{A}$. The arrangement of hyperplanes in $\mathcal{A}$ containing an edge $E$ in $L(\mathcal{A})$ is $$\mathcal{A}_E := \{H \in \mathcal{A}\ |\ E \subseteq H\}.$$
The hyperplane arrangement in the edge $E$ cut by $\mathcal{A}$ is
$$\mathcal{A}^E := \{H \cap E\ |\ H \in \mathcal{A},\, E \nsubseteq H\}.$$
A chamber of a hyperplane arrangement $\mathcal{A}$ is a connected component of the complement $\mathbb{R}^n \setminus \bigcup_{H \in \mathcal{A}} H$. Denote the set of all chambers of $\mathcal{A}$ by $\mathfrak{C}(\mathcal{A})$.

\smallskip

\noindent Assign a variable $a_H$ called weight to each hyperplane $H$ of an arrangement $\mathcal{A}$. Define the weight $\mathsf{a}(E)$ of an edge $E$ by $$\mathsf{a}(E) := \prod_{\substack{H \in \mathcal{A} \\ E \subseteq H}} a_H.$$

\noindent The multiplicity of an edge $E$ is $$l(E) := n(E)p(E),$$ where $n(E) := |\mathfrak{C}(\mathcal{A}^E)|$, and $p(E)$ is defined below.\\
For every edge $E$ of codimension $r$, let $N$ be an $r$-dimensional normal subspace to $E$. All hyperplanes of the resulting arrangement $(\mathcal{A}_E)^N$ pass through the point $\{v\} = E \cap N$. Consider the hyperplane arrangement which $(\mathcal{A}_E)^N$ induces in the tangent space $T_vN$. It determines another hyperplane arrangement $P\mathcal{A}_E$ in the projectivization of $T_vN$.\\ 
A chamber of an arrangement is bounded with respect to a hyperplane if the closure of this chamber does not intersect the hyperplane. For any arrangement $\mathcal{A}'$ in a real projective space, the numbers of chambers which are bounded with respect to its hyperplanes are all the same \cite[Theorem 1.5]{Va}, and we denote this number by $e(\mathcal{A}')$.\\
Finally, for an edge $E$ in $L(\mathcal{A})$, define $p(E) := e(P\mathcal{A}_E)$. 

\smallskip

\noindent Let $R_{\mathcal{A}} = \mathbb{Z}[a_H \,|\, H \in \mathcal{A}]$ be the ring of polynomials in variables $a_H$. The module of $R_{\mathcal{A}}$-linear combinations of chambers of the hyperplane arrangement $\mathcal{A}$ is
$$M_{\mathcal{A}} := \{ \sum_{C \in \mathfrak{C}(\mathcal{A})} x_C C\ |\ x_C \in R_{\mathcal{A}}\}.$$

\noindent Let $\mathcal{H}(C,D)$ be the set of hyperplanes separating the chambers $C$ and $D$ in $\mathfrak{C}(\mathcal{A})$. Define the $R_{\mathcal{A}}$-bilinear symmetric form $\mathsf{B}: M_{\mathcal{A}} \times M_{\mathcal{A}} \rightarrow R_{\mathcal{A}}$ by $$\mathsf{B}(C,C):=1,\ \text{and}\ \mathsf{B}(C,D):= \prod_{H \in \mathcal{H}(C,D)} a_H\ \text{if}\ C \neq D.$$
The Varchenko matrix of the hyperplane arrangement $\mathcal{A}$ is the matrix
$$\mathscr{M}_{\mathcal{A}} := \big(\mathsf{B}(C,D)\big)_{C,D \in \mathfrak{C}(\mathcal{A})}$$
of the bilinear symmetric form $\mathsf{B}$. The Varchenko determinant of the hyperplane arrangement $\mathcal{A}$ is the determinant $$\det \mathcal{A} := \det \mathscr{M}_{\mathcal{A}}.$$
The formula of this determinant due to Varchenko is \cite[(1.1) Theorem]{Va2}
$$\det \mathcal{A} = \prod_{E \in L(\mathcal{A})} \big(1- \mathsf{a}(E)^2\big)^{l(E)}.$$

\noindent It is, however, not feasible to directly use this formula to compute a determinant from a certain complexity level. For example, one can not deduce $\det \mathcal{B}_{24}$ directly from it. In this article, we show that we can work around this difficulty for certain hyperplane arrangements if we know the edges which are relevant. In this purpose, we use a clearer definition of the multiplicity $l(E)$ written in an article of Denham and Hanlon \cite[2. The Nullspace of the $B$ Matrices]{DeHa}: First choose a hyperplane $H$ containing $E$. Then $l(E)$ is half the number of chambers $C$ which have the property that $E$ is the minimal edge containing $\bar{C} \cap H$.\\
We determine the relevant edges in the next section. Then, we compute the Varchenko determinants of some hyperplane arrangements in the last section.

\section{The Relevant Edges}

\noindent In this section, we remove in the Varchenko determinant the factors $\big(1- \mathsf{a}(E)^2\big)^{l(E)}$ such that $l(E) = 0$. In the edge set $L(\mathcal{B}_7)$, for example, we do not need to consider the edges $\{x \in \mathbb{R}^7\ |\ x_1 = x_2,\, x_4 = x_5\}$ and $\{x \in \mathbb{R}^7\ |\ x_2 = x_4 = x_5,\, x_1 = x_7\}$ whose multiplicity is $0$. This removing simplifies the Varchenko determinant computing. 

\smallskip

\noindent Take a hyperplane arrangement $\mathcal{A}$ in $\mathbb{R}^n$. We say that an edge $E$ of $\mathcal{A}$ is relevant if $l(E) \neq 0$. Denote the relevant edge set of $\mathcal{A}$ by 
$$\mathfrak{R}_{\mathcal{A}} := \{E \in L_{\mathcal{A}}\ |\ l(E) \neq 0\}.$$

\noindent To determine $\mathfrak{R}_{\mathcal{A}}$, we have to consider the faces of the chambers. Recall that the face set of a chamber $C$ in $\mathfrak{C}_{\mathcal{A}}$ resp. of the chambers in $\mathfrak{C}_{\mathcal{A}}$ is $$\mathcal{F}(C) := \{\bar{C} \cap E \ |\ E \in L_{\mathcal{A}},\, \bar{C} \cap E \neq \emptyset\} \quad \text{resp.} \quad \mathcal{F}(\mathfrak{C}_{\mathcal{A}}) := \bigcup_{C \in \mathfrak{C}_{\mathcal{A}}} \mathcal{F}(C).$$
Define the following subset of $\mathcal{F}(C)$ resp. $\mathcal{F}(\mathfrak{C}_{\mathcal{A}})$
$$\mathcal{S}(C) := \{\bar{C} \cap H \ |\ H \in \mathcal{A},\, \bar{C} \cap H \neq \emptyset\} \quad \text{resp.} \quad \mathcal{S}(\mathfrak{C}_{\mathcal{A}}) := \bigcup_{C \in \mathfrak{C}_{\mathcal{A}}} \mathcal{S}(C).$$

\begin{lemma} \label{RG}
Let $\mathcal{A}$ be a hyperplane arrangement in $\mathbb{R}^n$. Then $$\mathfrak{R}_{\mathcal{A}} = \big\{\langle F \rangle\ |\ F \in \mathcal{S}(\mathfrak{C}_{\mathcal{A}}) \big\}.$$
\end{lemma}

\begin{proof}
Let $F \in \mathcal{S}(\mathfrak{C}_{\mathcal{A}})$. We have,
$$\langle F \rangle = \bigcap_{\substack{H \in \mathcal{A} \\ F \subseteq H}} H.$$
From definition, there exist a chamber $C$ and a hyperplane $H$ such that $\bar{C} \cap H = F$. Since $\langle F \rangle$ is the minimal edge containing $F$, then $l\big( \langle F \rangle \big) \geq 1$.

\smallskip

\noindent Now, take $E \in L_{\mathcal{A}} \setminus \big\{\langle F \rangle\ |\ F \in \mathcal{S}(\mathfrak{C}_{\mathcal{A}})\big\}$. Suppose that there exist a chamber $C$ and a hyperplane $H$ such that $E$ is the minimal edge containing $F = \bar{C} \cap H$. It means that $E \varsubsetneq \langle F \rangle$, which impossible since $F \subseteq E$.
\end{proof}

\begin{proposition} \label{PrRe}
Let $\mathcal{A}$ be a hyperplane arrangement in $\mathbb{R}^n$. For every relevant edge $E$, we fix a hyperplane $H_E$ of $\mathcal{A}$ containing it. Then,
\begin{align*}
& \det \mathcal{A} = \prod_{E \in \mathfrak{R}_{\mathcal{A}}} \big(1- \mathsf{a}(E)^2\big)^{l(E)} \\
& \text{with}\ l(E) = \frac{1}{2}\, \big| \big\{C \in \mathfrak{C}_{\mathcal{A}}\ |\ \langle \bar{C} \cap H_E \rangle = E \big\} \big|.
\end{align*}
\end{proposition}

\begin{proof}
It is clear that $$\prod_{E \in L_{\mathcal{A}}} \big(1- \mathsf{a}(E)^2\big)^{l(E)} = \prod_{E \in \mathfrak{R}_{\mathcal{A}}} \big(1- \mathsf{a}(E)^2\big)^{l(E)}.$$
Furthermore, from Lemma \ref{RG}, we deduce that $E$ is the minimal edge containing $\bar{C} \cap H_E$ if and only if $\langle \bar{C} \cap H_E \rangle = E$.
\end{proof}

\section{Some Varchenko Determinants}

\paragraph{Central $2$--Dimensional Arrangement.} It is a hyperplane arrangement $\mathcal{A} = \{H_1, \dots, H_m\}$ in $\mathbb{R}^2$ such that the intersection $\bigcap_{i=1}^m H_i$ is the origin $\{0\}$. It is the case of the hyperplane arrangement associated to the dihedral group $D_m$ having $2m$ elements with $$H_i = \big\{x \in \mathbb{R}^2\ |\ x_1 \cos \frac{(i-1)\pi}{m} + x_2 \sin \frac{(i-1)\pi}{m} = 0 \big\}.$$ Assign the weight $a_i$ to the hyperplane $H_i$. Then, $\mathfrak{R}_{\mathcal{A}} = \big\{H_1, \dots, H_m, \{0\}\big\}$, and
\begin{itemize}
\item $\mathsf{a}(H_i) = a_i$ with $l(H_i)=2$, 
\item $\mathsf{a}(\{0\}) = \prod_{i=1}^m a_i$ with $l(\{0\})=m-2$.  
\end{itemize}
Then, $$\det \mathcal{A} = \big(1- \prod_{i=1}^m a_i^2\big)^{m-2} \, \prod_{j=1}^m (1- a_j^2)^2.$$

\paragraph{General Position and the Hypercubic Arrangement.} A hyperplane arrangement $\mathcal{G}_n = \{H_1, \dots, H_{n+1}\}$ in $\mathbb{R}^n$ is in general position if, for every subset $P$ of $[n]$ such that $|P|=p$, we have $\dim \bigcap_{i \in P} H_i = n-p$. It is the case of the hyperplane arrangement such that
\begin{itemize}
\item $\forall i \in [n],\, H_i = \{x \in \mathbb{R}^n\ |\ x_i=0\}$,
\item $H_{n+1} = \{x \in \mathbb{R}^n\ |\ x_1 + \dots + x_n = 1\}$.
\end{itemize}  
Assign the weight $a_i$ to the hyperplane $H_i$.

\smallskip

\noindent And for $\alpha, \beta$ in $\mathbb{R}$ with $\alpha \neq \beta$, the hypercubic arrangement is the hyperplane arrangement $\mathcal{C}_n = \{H_{1,\alpha}, H_{1,\beta}, \dots, H_{n,\alpha}, H_{n,\beta}\}$ in $\mathbb{R}^n$ such that $H_{i,\alpha\, \text{resp.}\, \beta} = \{x \in \mathbb{R}^n\ |\ x_i = \alpha\ \text{resp.}\ \beta\}$.\\
Assign the weight $a_{i,\alpha}$ resp. $a_{i,\beta}$ to the hyperplane $H_{i,\alpha}$ resp. $H_{i,\beta}$.

\smallskip

\noindent Both hyperplane arrangements have the property
$$\forall H \in \mathcal{G}_n\ \text{resp.}\ \mathcal{C}_n,\, \forall C \in \mathfrak{C}(\mathcal{G}_n)\ \text{resp.}\ \mathfrak{C}(\mathcal{C}_n),\, \langle C \cap H \rangle = H\ \text{or}\ \langle C \cap H \rangle = \emptyset.$$
From Lemma \ref{RG}, we deduce that $\mathfrak{R}(\mathcal{G}_n) = \mathcal{G}_n$, and $\mathfrak{R}(\mathcal{C}_n) = \mathcal{C}_n$. 

\smallskip

\noindent Moreover, since $|\mathcal{G}_n| = 2^n-1$, and $|\mathcal{C}_n| = 3^n$, then 
$$\big|\{C \in \mathcal{G}_n\ |\ H \in \mathcal{G}_n,\, \langle C \cap H \rangle = H\}\big| = 2^n-2 \ \text{and}\ \big|\{C \in \mathcal{C}_n\ |\ H \in \mathcal{C}_n,\, \langle C \cap H \rangle = H\}\big| = 2 \times 3^{n-1}.$$

\noindent Thus $$\det \mathcal{G}_n = \prod_{i=1}^n (1- a_i^2)^{2^{n-1}-1} \ \text{and} \ \det \mathcal{C}_n = \prod_{i=1}^n (1- a_{i,\alpha}^2)^{3^{n-1}} (1- a_{i,\beta}^2)^{3^{n-1}}.$$

\paragraph{Braid Arrangement.} It consists of the $\binom{n}{2}$ hyperplanes $H_{i,j} = \{x \in \mathbb{R}^n \ |\ x_i - x_j=0\}$, with $1 \leq i < j \leq n$. We assign the weight $a_{i,j}$ to the hyperplane $H_{i,j}$.

\begin{proposition} \label{detAn}
Let $n \geq 2$. We have $$\det \mathcal{B}_n = \prod_{\substack{I \in 2^{[n]} \\ |I| \geq 2}} \Big(1- \prod_{\{i,j\} \in \binom{I}{2}} a_{i,j}^2\Big)^{(|I|-2)!\,(n-|I|+1)!}.$$
\end{proposition} 

\noindent This determinant was also calculated by Duchamp et al. \cite[6.4.2 A Decomposition of $B_n$]{DuEtAl} using the diagonal solutions of the Yang-Baxter equation. 

\smallskip

\noindent Each chamber of $\mathcal{B}_n$ is defined by $\{x \in \mathbb{R}^n\ |\ x_{\sigma(1)} > x_{\sigma(2)} > \dots > x_{\sigma(n)}\}$, where $\sigma$ is a permutation of $[n]$. We write $\{x_{\sigma(1)} > x_{\sigma(2)} > \dots > x_{\sigma(n)}\}$ for simplicity.\\
Let $I = \{i_1, \dots, i_r\}$ be a subset of $[n]$, with $|I| \geq 2$. Denote by $E(I)$ the edge
$$E(I) := \bigcap_{\{i,j\} \in \binom{\{i_1, \dots, i_r\}}{2}} H_{i,j}.$$

\begin{lemma}
Let $n \geq 2$. We have $\mathfrak{R}(\mathcal{B}_n) = \big\{E(I)\ |\ I \subseteq [n],\, |I| \geq 2\big\}$.
\end{lemma}

\begin{proof}
Consider a hyperplane $H_{s,t}$ of $\mathcal{B}_n$, and a permutation $\sigma$ of $[n]$ such that $\sigma(i) = s$ and $\sigma(j) = t$ with $i < j$. Then,
\begin{align*}
\big\langle H_{s, t} \cap \overline{\{x_{\sigma(1)} > x_{\sigma(2)} > \dots > x_{\sigma(n)}\}} \big\rangle
& = \big\langle \{x_{\sigma(1)} > \dots > x_{\sigma(i)} = \dots = x_{\sigma(j)} > \dots > x_{\sigma(n)}\} \big\rangle \\
& = \{x \in \mathbb{R}^n\ |\ x_{\sigma(i)} = x_{\sigma(i+1)} = \dots = x_{\sigma(j)}\} \\
& = E\big(\{\sigma(i),\, \sigma(i+1),\, \dots,\, \sigma(j)\}\big).
\end{align*}
Hence, $E\big(\{\sigma(i),\, \sigma(i+1),\, \dots,\, \sigma(j)\}\big)$ is the minimal edge containing the face\\ $H_{i_s, i_t} \cap \overline{\{x_{\sigma(1)} > x_{\sigma(2)} > \dots > x_{\sigma(n)}\}}$.
\end{proof}

\begin{lemma} \label{LeAn}
Let $I \subseteq [n]$. Then, $l\big(E(I)\big) = (|I|-2)!\,(n-|I|+1)!$.
\end{lemma}

\begin{proof} Let $H_{i_1, i_r}$ be a hyperplane containing $E(I)$. We have to count the chambers $\{x_{\sigma(1)} > x_{\sigma(2)} > \dots > x_{\sigma(n)}\}$ such that $\big\langle H_{i_1, i_r} \cap \overline{\{x_{\sigma(1)} > x_{\sigma(2)} > \dots > x_{\sigma(n)}\}} \big\rangle = E(I)$.\\
Let $\nu$ be a permutation of $\{2, \dots, r-1\}$. These chambers correspond to the chambers having the forms
\begin{align*}
& \{\dots > x_{i_1} > x_{i_{\nu(2)}} > \dots > x_{i_{\nu(r-1)}} > x_{i_r} > \dots\} \\
&\text{and}\ \{\dots > x_{i_r} > x_{i_{\nu(2)}} > \dots > x_{i_{\nu(r-1)}} > x_{i_1} > \dots\}.
\end{align*}
Because of the coefficient $\frac{1}{2}$ in the multiplicity, we just need to consider the chambers
$$\{x_{\sigma(1)} > x_{\sigma(2)} > \dots > x_{\sigma(n)}\} =  \{\dots > x_{i_1} > x_{i_{\nu(2)}} > \dots > x_{i_{\nu(r-1)}} > x_{i_r} > \dots\}.$$
Let $i \in [n]$ such that $\sigma(i) = i_1$. We have:
\begin{itemize}
\item[$\bullet$] $(n-r)!$ possibilities for the sequence $\big(\sigma(1),\, \dots,\, \sigma(i-1),\, \sigma(i+r),\, \dots,\,  \sigma(n)\big)$,
\item[$\bullet$] $(r-2)!$ possibilities for the sequence $\big(\sigma(i+1),\, \dots,\, \sigma(i+r-2)\big)$,
\item[$\bullet$] and $n-r+1$ possibilities to choose $i$ since we must have $i \in [n-r+1]$.
\end{itemize}
Then $l\big( E(I) \big) = (n-r)! \times (r-2)! \times (n-r+1) = (r-2)!\,(n-r+1)!$.
\end{proof}

\noindent We obtain Proposition \ref{detAn} by combining Proposition \ref{PrRe} and Lemma \ref{LeAn}.

\paragraph{Hyperplane Arrangement Associated to Hyperoctahedral Group.} Let $[\pm n] := \{-n, \dots, -2, -1, 1, 2, \dots, n\}$. Denote $\overline{2^{[\pm n]}}$ the subset of $2^{[\pm n]}$ having the following properties:
\begin{itemize}
\item[$\bullet$] the elements of $\overline{2^{[\pm n]}}$ are the elements $\{i_1, \dots, i_t\}$ of $2^{[\pm n]}$ such that $|i_r| \neq |i_s|$ if $r \neq s$,
\item[$\bullet$] and if $\{i_1, \dots, i_t\} \in \overline{2^{[\pm n]}}$, then $\{-i_1, \dots, -i_t\} \notin \overline{2^{[\pm n]}}$.
\end{itemize}
For example,
\begin{align*}
2^{[\pm 3]} = & \big\{\emptyset,\, \{1\},\, \{2\},\, \{3\},\\
& \{1,2\},\, \{-1,2\},\, \{1,3\},\, \{-1,3\},\, \{2,3\},\, \{-2,3\},\\
& \{1,2,3\},\, \{-1,2,3\},\, \{-1,-2,3\},\, \{1,-2,3\}\big\}.
\end{align*}

\noindent The hyperplane arrangement $\mathcal{O}_n$ associated to the hyperoctahedral group $B_n$ consists of
\begin{itemize}
\item the $\binom{n}{2}$ hyperplanes $H_{i,j} = \{x \in \mathbb{R}^n \ |\ x_i - x_j=0\}$ with $1 \leq i < j \leq n$,
\item the $\binom{n}{2}$ hyperplanes $H_{-i,j} = \{x \in \mathbb{R}^n \ |\ x_i+x_j=0\}$ with $1 \leq i < j \leq n$,
\item the $n$ hyperplanes $H_i = \{x \in \mathbb{R}^n \ |\ x_i = 0\}$ with $i \in [n]$.
\end{itemize}
We assign the weights $a_{i,j}$ to the hyperplanes $H_{i,j}$, the weights $a_{-i,j}$ to the hyperplanes $H_{-i,j}$, and the weights $a_i$ to the hyperplanes $H_i$. 

\begin{proposition} \label{DetBn}
Let $n \geq 2$. We have 
\begin{align*}
\det \mathcal{O}_n = & \prod_{\substack{J \in \overline{2^{[\pm n]}} \\ |J| \geq 2}} \Big(1- \prod_{\{i,j\} \in \binom{J}{2}} a_{i,j}^2\Big)^{2^{n-|J|+1}\,(|J|-2)!\,(n-|J|+1)!} \\
& \prod_{\substack{I \in 2^{[n]} \\ |I| \geq 1}} \Big(1- \prod_{i \in I} a_i^2
\prod_{\{i,j\} \in \binom{I}{2}} a_{i,j}^2\, a_{-i,j}^2\Big)^{2^{n-1}\,(|I|-1)!\,(n-|I|)!}.
\end{align*}
\end{proposition}

\noindent Each chamber of $\mathcal{O}_n$ is defined by $\{x \in \mathbb{R}^n\ |\ \epsilon_1 x_{\sigma(1)} > \epsilon_2 x_{\sigma(2)} > \dots > \epsilon_n x_{\sigma(n)} > 0\}$, where $\epsilon_i \in [\pm 1]$ and $\sigma$ is a permutation of $[n]$. For $J = \{\epsilon_1 i_1,\, \epsilon_2 i_2,\, \dots,\, \epsilon_r i_r\} \in \overline{2^{[\pm n]}}$, with $|J| \geq 2$, denote $E(J)$ the edge $$E(J) := \{x \in \mathbb{R}^n\ |\ \epsilon_1 x_{i_1} = \epsilon_2 x_{i_2} = \dots = \epsilon_r x_{i_r}\}.$$
And for $I =\{i_1, i_2, \dots, i_r\} \subseteq 2^{[n]}$, with $r \geq 1$, denote $E(I_0)$ the edge
$$E(I_0) := \{x \in \mathbb{R}^n\ |\ x_{i_1} = x_{i_2} = \dots = x_{i_r} = 0\}.$$

\begin{lemma}
Let $n \geq 2$. We have $$\mathfrak{R}(\mathcal{O}_n) = \big\{E(J)\ |\ J \in \overline{2^{[\pm n]}},\, |J| \geq 2\big\} \, \cup \, \big\{E(I_0)\ |\ I \in 2^{[n]},\, |I| \geq 1\big\}.$$
\end{lemma}

\begin{proof}
Consider a hyperplane $H_{\epsilon_s s,\epsilon_t t}$ of $\mathcal{O}_n$, and a permutation $\sigma$ of $[n]$ such that $\sigma(i) = s$ and $\sigma(j) = t$ with $i < j$. If
\begin{itemize}
\item[$\bullet$] $\{\epsilon_s s, \epsilon_t t\} = \{\epsilon_i \sigma(i), \epsilon_j \sigma(j)\}$, then
\begin{align*}
\big\langle H_{\epsilon_s s, \epsilon_t t} \cap & \, \overline{\{\epsilon_1 x_{\sigma(1)} > \epsilon_2 x_{\sigma(2)} > \dots > \epsilon_n x_{\sigma(n)} > 0\}} \big\rangle \\
& = \big\langle \{\epsilon_1 x_{\sigma(1)} > \dots > \epsilon_i x_{\sigma(i)} = \dots = \epsilon_j x_{\sigma(j)} > \dots > \epsilon_n x_{\sigma(n)} > 0\} \big\rangle \\
& = \{\epsilon_i x_{\sigma(i)} = \epsilon_{i+1} x_{\sigma(i+1)} = \dots = \epsilon_j x_{\sigma(j)}\} \\
& = E\big(\{\epsilon_i \sigma(i),\, \epsilon_{i+1} \sigma(i+1),\, \dots,\, \epsilon_j \sigma(j)\}\big).
\end{align*}
\item[$\bullet$] $\{\epsilon_s i_s, \epsilon_t i_t\} \neq \{\epsilon_i \sigma(i), \epsilon_j \sigma(j)\}$, then
\begin{align*}
\big\langle H_{\epsilon_s s, \epsilon_t t} \cap & \, \overline{\{\epsilon_1 x_{\sigma(1)} > \epsilon_2 x_{\sigma(2)} > \dots > \epsilon_n x_{\sigma(n)} > 0\}} \big\rangle \\
& = \big\langle \{\epsilon_1 x_{\sigma(1)} > \dots > \epsilon_{i-1} x_{\sigma(i-1)} > x_{\sigma(i)} = \dots = x_{\sigma(n)} = 0\} \big\rangle \\
& = \{x_{\sigma(i)} = x_{\sigma(i+1)} = \dots = x_{\sigma(n)} = 0\} \\
& = E\big(\{0,\, \sigma(i),\, \sigma(i+1),\, \dots,\, \sigma(n)\}\big).
\end{align*}
\end{itemize}
Consider a hyperplane $H_u$ of $\mathcal{O}_n$, and a permutation $\sigma$ of $[n]$ such that $\sigma(i) = u$. Then,
\begin{align*}
\big\langle H_u \cap & \, \overline{\{\epsilon_1 x_{\sigma(1)} > \epsilon_2 x_{\sigma(2)} > \dots > \epsilon_n x_{\sigma(n)} > 0\}} \big\rangle \\
& = \big\langle \{\epsilon_1 x_{\sigma(1)} > \dots > \epsilon_{i-1} x_{\sigma(i-1)} > x_{\sigma(i)} = \dots = x_{\sigma(n)} = 0\} \big\rangle \\
& = \{x_{\sigma(i)} = x_{\sigma(i+1)} = \dots = x_{\sigma(n)} = 0\} \\
& = E\big(\{0,\, \sigma(i),\, \sigma(i+1),\, \dots,\, \sigma(n)\}\big).
\end{align*}
\end{proof}

\begin{lemma} \label{LeBn}
Let $E(J),\, E(I_0) \in \mathfrak{R}(O_n)$. Then,
\begin{align*}
& \mathsf{a}\big(E(J)\big) = \prod_{\{s,t\} \in \binom{J}{2}} a_{s,t} \quad \text{with} \quad l\big(E(J)\big) = 2^{n-|J|+1}(|J|-2)!\,(n-|J|+1)!,\\
& \mathsf{a}\big(E(I_0)\big) = \prod_{u \in I} a_u \prod_{\{s,t\} \in \binom{I}{2}} a_{s,t} a_{-s,t} \quad \text{with} \quad l\big(E(I_0)\big) = 2^{n-1}(|I|-1)!\,(n-|I|)!.
\end{align*}
\end{lemma}

\begin{proof}
We have $$\mathsf{a}\big(E(J)\big) = \prod_{\substack{H \in \mathcal{O}_n \\ E(J) \subseteq H}} \mathsf{a}(H) = \prod_{\{s,t\} \in \binom{J}{2}} a_{s,t},$$
and $$\mathsf{a}\big(E(I_0)\big) = \prod_{\substack{H \in \mathcal{O}_n \\ E(I_0) \subseteq H}} \mathsf{a}(H) = \prod_{u \in I} a_u \prod_{\{s,t\} \in \binom{I}{2}} a_{s,t} a_{-s,t}.$$
To the edge $E(J)$, assign the hyperplane $H_{\epsilon_1 i_1,\, \epsilon_r i_r}$ containing it. We first have to count the chambers $\{\epsilon_1 x_{\sigma(1)} > \epsilon_2 x_{\sigma(2)} > \dots > \epsilon_n x_{\sigma(n)} > 0\}$ such that
$$\big\langle H_{\epsilon_1 i_1,\, \epsilon_r i_r} \cap \overline{\{\epsilon_1 x_{\sigma(1)} > \epsilon_2 x_{\sigma(2)} > \dots > \epsilon_n x_{\sigma(n)} > 0\}} \big\rangle = E(J).$$
Let $\nu$ be a permutation of $\{2, \dots, r-1\}$. These chambers correspond to the chambers having the forms
\begin{align*}
& \{\dots > \epsilon_1 x_{i_1} > \epsilon_{\nu(2)} x_{i_{\nu(2)}} > \dots > \epsilon_{\nu(r-1)} x_{i_{\nu(r-1)}} > \epsilon_r x_{i_r} > \dots\}, \\
& \{\dots > \epsilon_r x_{i_r} > \epsilon_{\nu(2)} x_{i_{\nu(2)}} > \dots > \epsilon_{\nu(r-1)} x_{i_{\nu(r-1)}} > \epsilon_1 x_{i_1} > \dots\}, \\
& \{\dots > -\epsilon_1 x_{i_1} > \epsilon_{\nu(2)} x_{i_{\nu(2)}} > \dots > \epsilon_{\nu(r-1)} x_{i_{\nu(r-1)}} > -\epsilon_r x_{i_r} > \dots\}, \\
\text{and} \quad & \{\dots > -\epsilon_r x_{i_r} > \epsilon_{\nu(2)} x_{i_{\nu(2)}} > \dots > \epsilon_{\nu(r-1)} x_{i_{\nu(r-1)}} > -\epsilon_1 x_{i_1} > \dots\}.
\end{align*}
Because of the coefficient $\frac{1}{2}$ in the multiplicity and the symmetry of the signed permutation, we just need to consider the chambers
\begin{align*}
\{\epsilon_1 x_{\sigma(1)} > & \, \epsilon_2 x_{\sigma(2)} > \dots > \epsilon_n x_{\sigma(n)} > 0\} \\
& = \{\dots > \epsilon_1 x_{i_1} > \epsilon_{\nu(2)} x_{i_{\nu(2)}} > \dots > \epsilon_{\nu(r-1)} x_{i_{\nu(r-1)}} > \epsilon_r x_{i_r} > \dots\}
\end{align*}
and multiply the obtained cardinality by $2$.
Let $i \in [n]$ such that $\sigma(i) = i_1$. We have:
\begin{itemize}
\item[$\bullet$] $2^{n-r}(n-r)!$ possibilities for the sequence $\big(\epsilon_1 \sigma(1),\, \dots,\, \epsilon_{i-1} \sigma(i-1),\, \epsilon_{i+r} \sigma(i+r),\, \dots,\, \epsilon_n \sigma(n)\big)$,
\item[$\bullet$] $(r-2)!$ possibilities for the sequence $\big(\epsilon_{i+1} \sigma(i+1),\, \dots,\, \epsilon_{i+r-2} \sigma(i+r-2)\big)$,
\item[$\bullet$] and $n-r+1$ possibilities to choose $i$ since we must have $i \in [n-r+1]$.
\end{itemize}
Then $l\big(E(J)\big) = 2 \times 2^{n-r}(n-r)! \times (r-2)! \times (n-r+1) = 2^{n-r+1}(r-2)!(n-r+1)!$.

\smallskip

\noindent To the edge $E(I_0)$, assign the hyperplane $H_{i_r}$ containing it. Now, we have to count the chambers $\{\epsilon_1 x_{\sigma(1)} > \epsilon_2 x_{\sigma(2)} > \dots > \epsilon_n x_{\sigma(n)} > 0\}$ such that
$$\big\langle H_{i_r} \cap \overline{\{\epsilon_1 x_{\sigma(1)} > \epsilon_2 x_{\sigma(2)} > \dots > \epsilon_n x_{\sigma(n)} > 0\}} \big\rangle = E(I_0).$$
Let $\nu$ be a permutation of $[r-1]$. Those chambers correspond to the chambers having the forms
\begin{align*}
& \{\dots > x_{i_r} > \epsilon_{\nu(r-1)} x_{i_{\nu(r-1)}} > \dots > \epsilon_{\nu(2)} x_{i_{\nu(2)}} > \epsilon_{\nu(1)} x_{i_{\nu(1)}} > 0\} \\
\text{and} \quad & \{\dots > - x_{i_r} > \epsilon_{\nu(r-1)} x_{i_{\nu(r-1)}} > \dots > \epsilon_{\nu(2)} x_{i_{\nu(2)}} > \epsilon_{\nu(1)} x_{i_{\nu(1)}} > 0\}.
\end{align*}
Because of the coefficient $\frac{1}{2}$ in the multiplicity, we just need to consider the chambers
\begin{align*}
\{\epsilon_1 x_{\sigma(1)} > & \, \epsilon_2 x_{\sigma(2)} > \dots > \epsilon_n x_{\sigma(n)} > 0\} \\
& = \{\dots > x_{i_r} > \epsilon_{\nu(r-1)} x_{i_{\nu(r-1)}} > \dots > \epsilon_{\nu(2)} x_{i_{\nu(2)}} > \epsilon_{\nu(1)} x_{i_{\nu(1)}} > 0\}.
\end{align*}
We have:
\begin{itemize}
\item[$\bullet$] $2^{n-r}(n-r)!$ possibilities for the sequence $\big(\epsilon_1 \sigma(1),\, \dots,\, \epsilon_{n-r} \sigma(n-r)\big)$,
\item[$\bullet$] and $2^{r-1}(r-1)!$ possibilities for the sequence $\big((\epsilon_{n-r+2} \sigma(n-r+2),\, \dots,\, \epsilon_n \sigma(n)\big)$.
\end{itemize}
Then $l\big(E(I_0)\big) = 2^{n-r}(n-r)! \times 2^{r-1}(r-1)! = 2^{n-1}(r-1)!(n-r)!$.
\end{proof}

\noindent We obtain Proposition \ref{DetBn} by combining Proposition \ref{PrRe} and Lemma \ref{LeBn}.

\bibliographystyle{abbrvnat}

\end{document}